\documentclass[12pt]{article}
\usepackage{array,epsfig}
\usepackage{amsmath}
\usepackage{amsfonts}
\usepackage{amssymb}
\usepackage{amsxtra}
\usepackage{amsthm}
\usepackage{mathrsfs}
\usepackage{color}
\usepackage{float}
\usepackage{setspace}
\usepackage{hyperref}
\hypersetup{allcolors=black}
\theoremstyle{definition}

\newtheorem{thm}{Theorem}

\newcommand{\ra}{\rightarrow}

\newcommand{\xra}{\xrightarrow}

\newcommand{\bb}[1]{\mathbb{#1}}
\def\ba#1\ea{\begin{aligned}#1\end{aligned}}
\newcommand{\Z}{\bb{Z}}

\newcommand{\R}{\bb{R}}
\newcommand{\C}{\bb{C}}

\newcommand{\equivclass}[1]{#1/{\sim}}

\DeclareMathOperator{\Gr}{Gr}
\DeclareMathOperator{\GL}{GL}

\DeclareMathOperator{\Hom}{Hom}

\DeclareMathOperator{\Ext}{Ext}

\DeclareMathOperator{\SO}{SO}
\DeclareMathOperator{\Ort}{O}
\DeclareMathOperator{\SU}{SU}
\DeclareMathOperator{\Spin}{Spin}
\DeclareMathOperator{\Flag}{Flag}

\providecommand{\keywords}[1]{{\textit{Keywords:} #1}}
\usepackage{pdfpages}

\newcommand{\Address}{{
		\bigskip
		\footnotesize
		A.B.~Yetişer, \textsc{Department of Mathematics, National Research University Higher School of Economics, Moscow, Russia}\par\nopagebreak
		\textit{E-mail address:} \texttt{ayetisher@edu.hse.ru, abyetiser@gmail.com}
}}
\begin{document}
\title{The space of all triples of projective lines of distinct intersections in $\mathbb{RP}^n$}
\date{}
\author{Ali Berkay Yetişer}
\maketitle
\setcounter{page}{1}
\begin{abstract}
	We study the space of all triples of projective lines in $\mathbb{RP}^n$ such that any line in a triple intersects the two others at distinct points. We show that for $n=2$ and $3$ these spaces are homotopically equivalent to the real complete flag variety $Flag(\mathbb{R}^n)$ for $n=3$ and $4,$ respectively, and we explicitly calculate the integral homology of the corresponding spaces. We prove that for arbitrary $n$, this space is homotopy equivalent to $Flag(1,2,3,\mathbb{R}^{n+1}),$ the variety of all partial flags of signature $(0,1,2,3,n+1)$ in an $(n+1)$-dimensional vector space over $\mathbb{R}.$
	
	\vspace{2mm}
	
	\noindent\keywords{flag variety, homogeneous space, homotopy equivalence, projective configuration}
\end{abstract}

\section{Introduction}\label{sec1}
We wish to study the topology of the space $\mathcal H$ of all triples of projective lines in $\bb {RP}^3$ such that any line in a triple intersects the two others at distinct points. A triple of such lines spans a unique projective plane in $\bb{RP}^3.$ We are also interested in the corresponding problem in $\bb{RP}^2$ because it will come up as the fibers of a fibration $\mathcal H\ra \bb{RP}^3,$ and the problem is more straightforward in $\bb {RP}^2$ because the corresponding matrices have full rank. Hence, we will first study the space $\mathcal H'$ of all triples of lines $(\ell_1,\ell_2,\ell_3)$ in $\bb {RP}^2$ with the same condition, that is, any line in a triple intersects the two others at distinct points. In Section \ref{sec2}, we will prove the following results.
\newtheorem*{thm1}{Theorem \ref{thm1}}
\begin{thm1}
	The space $\mathcal H'$ is homotopy equivalent to the quotient of $\SU(2)$ by the free action of the dihedral group of order $8,$ $D_8.$
\end{thm1}

\newtheorem*{thm2}{Theorem \ref{thm2}}
\begin{thm2}
	The integral homology groups of $\mathcal H'$ are as follows. $$H_0(\mathcal H')\cong H_3(\mathcal H')\cong \Z, H_1(\mathcal H')\cong \Z/2\Z\times \Z/2\Z,$$ and for any other subscript the homology is trivial.  
\end{thm2} 

In Section \ref{sec3}, we will prove that the space $\mathcal H$ is homotopically equivalent to a complete flag variety.
\newtheorem*{thm3}{Theorem \ref{thm3}}
\begin{thm3}
	The space $\mathcal H$ is homotopy equivalent to $\Flag(\R^4).$
\end{thm3}
We will explicitly calculate the integral homology groups of $\mathcal H$ using the theory of homogeneous spaces and the Serre spectral sequence associated to the fiber bundle $\pi:\mathcal H\ra \bb{RP}^3.$ 

\newtheorem*{thm4}{Theorem \ref{thm4}}
\begin{thm4}
	The integral homology groups of $\mathcal H\cong \Flag(\R^4)$ are as follows. $$H_0(\mathcal H)\cong H_6(\mathcal H)\cong\Z, H_1(\mathcal H)\cong H_4(\mathcal H)\cong (\Z/2\Z)^3,$$ $$H_2(\mathcal H)\cong (\Z/2\Z)^2, H_3(\mathcal H)\cong \Z^2\times (\Z/2\Z)^2,$$ and for any other subscript homology is trivial.
\end{thm4}

In Section \ref{sec4}, we will generalize the homotopy equivalence between $\mathcal H$ and $\Flag(\R^4)$ to a homotopy equivalence between the space of triples of projective lines of distinct intersections in $\bb{RP}^n$ and the partial flag variety $\Flag(1,2,3,\R^{n+1})$. 
\newtheorem*{thm5}{Theorem \ref{thm5}}
\begin{thm5}
	The space of all triples of projective lines in $\bb{RP}^{n}$ such that any line in a triple intersects the two others at distinct points is homotopy equivalent to the real partial flag variety $\Flag(1,2,3,\R^{n+1})$ consisting of partial flags of signature $(0,1,2,3,n+1)$.
\end{thm5}

\section{The space $\mathcal{H}'$}\label{sec2}
The projective lines $\ell_i\subset \bb{RP}^2$ correspond to planes passing through the origin in $\R^3,$ $\sum_{j=1}^3 a_{ij}x_j=0$ and a plane passing through the origin in $\R^3$ can be described with a normal vector to the plane, whose components are the coefficients $a_{ij}.$ The intersection condition, which says that any two planes intersect at a distinct line, is equivalent to saying that the three normal vectors are linearly independent which translates to the fact that the rows of the matrix $A=[a_{ij}]$ are linearly independent. Since the rows of the matrix $A$ correspond to the normal vectors, if a normal vector is scaled by a non-zero scalar $\lambda \in \R$ it describes the same plane passing through the origin. Hence the space $\mathcal H'$ is the quotient $\equivclass{\GL(3,\R)}$ where the equivalence relation $\sim$ is defined as follows. For $A,B\in \GL(3,\R),$ $A\sim B$ if for any $i=1,2,3$ there exists a non-zero scalar $\lambda_i\in \R$ such that $A_i=\lambda_i B_i$ where $A_i$ is the $i$th row of the matrix $A$ and similarly $B_i$ is the $i$th row of the matrix $B.$ We may write this quotient space as $\GL(3,\R)/T$ where $T$ is the maximal torus in $\GL(3,\R)$ consisting of diagonal matrices.

\begin{thm}\label{thm1}
	The space $\mathcal H',$ that is, the quotient $\equivclass{\GL(3,\R)} = \GL(3,\R)/T$ is homotopy equivalent to the quotient of $\SU(2)$ by the free action of the dihedral group of order $8,$ $D_8.$
\end{thm}
\begin{proof}
	By using QR factorization we can write $\GL(3,\R) = \Ort(3,\R) \boldsymbol\cdot U$ where $\Ort(3,\R)$ is the subgroup of $\GL(3,\R)$ consisting of $3$ by $3$ orthogonal matrices with real entries and $U$ is the subgroup of $\GL(3,\R)$ consisting of upper-triangular matrices with positive diagonal entries. 
	
	The deformation retraction of $\GL(3,\R)$ to $\Ort(3,\R)$ induces a deformation retraction of $\equivclass{\GL(3,\R)} =  \GL(3,\R)/T$ to $\equivclass{\Ort(3,\R)}=\Ort(3,\R)/(T\cap \Ort(3,\R)).$ Note that $T\cap \Ort(3,\R)$ consists of diagonal matrices with entries $\pm 1$.
	
	The subgroup $\Ort(3,\R)$ has two connected components and there is an element of $T\cap \Ort(3,\R)$ with negative determinant that can take an element of the connected component consisting of orthogonal matrices with determinant $-1$ to the other connected component $\SO(3,\R)$, so we see that $\Ort(3,\R)/(T\cap \Ort(3,\R))$ is diffeomorphic to $\SO(3,\R)/(T\cap \SO(3,\R))$ where $T\cap \SO(3,\R)$ is isomorphic to the Klein $4$-group $V_4\cong \Z/2\Z \times \Z/2\Z.$ 
	
	Now the universal covering group of $\SO(3,\R)$ is $\SU(2)$ with the kernel of the universal covering map $\pi:\SU(2)\ra \SO(3,\R)$ being $Q:=\ker \pi \cong \Z/2\Z.$ The lifting of $K:=T\cap \SO(3,\R)$ to the covering $\SU(2)$ is the nontrivial central extension $K^*$ of $K$ by the center of $\SU(2)$ which is isomorphic to the dihedral group of order $8,$ $D_8,$ we have two short exact sequences $$1\rightarrow Q (\cong \mathbb Z / 2\Z)\rightarrow \SU(2)\rightarrow \SO(3)\rightarrow 1$$ $$1\rightarrow Q(\cong\Z /2\Z)\rightarrow K^* (\cong D_8)\rightarrow  K (\cong \Z/ 2\Z \times \Z/2\Z)\rightarrow 1$$ and $Q$ and $K^*$ are normal in $\SU(2),$ and since $\SU(2)/Q\cong \SO(3,\R),$ we get $\SU(2)/K^* \cong \SO(3,\R)/K\cong \mathcal{H}'.$
\end{proof}

By definition a \textit{flag} in an $n$-dimensional vector space $V$ over a field $k$ is a sequence of subspaces of $V,$ $0=V_0\subset V_1\subset \cdots V_m=V$ such that $d_i=\dim V_i$ and $0=d_0<d_1<\cdots <d_m=n.$ A flag is called a \textit{complete flag} if $d_i=i$ for all $i.$ A \textit{standard flag} associated to an ordered basis of $V$ is a flag in which $i$th subspace is spanned by the first $i$ basis vectors. Let us call the variety of complete flags in an $n$-dimensional vector space over real numbers, $\Flag(\R^n).$ 

Note that $\mathcal H'$ is a homogeneous space and it is homotopy equivalent to the variety of complete flags in a three-dimensional vector space over the real numbers, $\Flag(\R^3)$. The general linear group $\GL(3,\R)$ acts transitively on the set of all complete flags in $\R^3,$ and with respect to a fixed basis the stabilizer of the standard flag is the group of nonsingular lower triangular matrices $B_3$, and so the real complete flag variety is the quotient $\GL(3,\R)/B_3$ which is again homotopy equivalent to the space $\Ort(3,\R)/(T\cap \Ort(3,\R))$ described in Theorem \ref{thm1}. 

We will calculate the integral homology groups of $\mathcal H'$ using Theorem \ref{thm1}. Note that the homology of the corresponding real complete flag variety is calculated in \cite{6}, and so it constitutes a verification of the result of the next theorem.

\begin{thm}\label{thm2}
	The integral homology groups of $\mathcal H'\cong \equivclass{\GL(3,\R)}\cong \SU(2)/K^*,$  are as follows. $$H_0(\mathcal H')\cong H_3(\mathcal H')\cong \Z, H_1(\mathcal H')\cong \Z/2\Z\times \Z/2\Z,$$ and for any other subscript the homology is trivial.  
\end{thm} 
\begin{proof}
	Let us compute the integral homology of $\SU(2)/K^*\cong \SO(3)/K$. Since the equivalence relation acts smoothly and freely on the smooth manifold $\SO(3)$, the dimension is $\dim(\SO(3,\R))-\dim K=3-0=3.$ The quotient is compact and connected because $\SO(3)$ is compact and connected. For any $\gamma \in K,$ the diffeomorphism $\gamma\cdot -:\SO(3)\ra \SO(3)$ which is defined as $x\mapsto \gamma x$ is orientation-preserving, so $K$ is orientation-preserving, hence the quotient is orientable because $\SO(3,\R)$ is a closed, oriented, smooth manifold and $K$ is a discrete group acting freely and properly on $\SO(3,\R).$ Thus, $H_0(\mathcal H')\cong H_3(\mathcal H')\cong \Z$. 
	
	By the Hurewicz theorem $H_1(\SU(2)/K^*)$ is isomorphic to the abelianization of $\pi_1(\SU(2)/K^*)\cong K^*\cong D_8$ which is the Klein $4$-group $\Z/2\Z\times \Z/2\Z.$ Hence, $H_1(\mathcal H')\cong \Z/2\Z \times \Z/2\Z.$ 
	
	Finally, by the Poincaré duality $H_2(\SU(2)/K^*)$ is isomorphic to $H^1(\SU(2)/K^*).$ By the universal coefficients theorem for cohomology and since $$\Ext^1(H_0(\SU(2)/K^*),\Z)=\Ext^1(\Z,\Z)$$ is trivial we have $$H_2(\mathcal H')\cong H_2(\SU(2)/K^*)\cong H^1(\SU(2)/K^*)\cong \Hom(H_1(\SU(2)/K^*),\Z)\cong 0.$$
\end{proof}
We could have taken a possible shortcut in the calculation of the homology groups of $\mathcal H'$ by the fact that this space is homotopy equivalent to a fiber bundle over $\bb{RP}^2$ with fibers $\bb{RP}^1$ homeomorphic to $S^1.$ 

\section{The space $\mathcal H$}\label{sec3}
Let us now consider the space $\mathcal H$ of all triples of lines in $\bb {RP}^3$ such that any line in a triple intersects the other two others at distinct points. We can homotopically identify $\mathcal H$ with the variety of complete flags in a four-dimensional vector space over the real numbers, $\Flag(\R^4)$. 

\begin{thm}\label{thm3}
	The space $\mathcal H$ is homotopy equivalent to $\Flag(\R^4).$
\end{thm}
\begin{proof}
	Let $p=(\ell_1,\ell_2,\ell_3)$ be a point in $\mathcal H.$ Each $\ell_i$ corresponds to a $2$-subspace of $V=\R^4,$ and $\ell_1\cup \ell_2\cup \ell_3$ is contained in some $3$-subspace $V'$ of $V.$ Let $e_i$ be a vector orthogonal to $\ell_i$ in $V',$ and let $V_1=\langle e_1 \rangle,$ $V_2=\langle e_1,e_2 \rangle,$ $V_3=\langle e_1,e_2,e_3\rangle,$ and so $(0=V_0<V_1<V_2<V_3<V)$ is a complete flag in $\R^4$ and we define $f:\mathcal H \ra \Flag(\R^4)$ to be the map that takes $p$ to this complete flag. 
		
	A complete flag $(0=V_0<V_1<V_2<V_3<V=\R^4)$ has a unique orthonormal basis $(e_1,e_2,e_3,e_4)$ up to multiplying each basis element by a unit such that $0=V_0<V_1=\langle e_1\rangle < V_2=\langle e_1,e_2\rangle < V_3=\langle e_1,e_2,e_3\rangle < V=\langle e_1,e_2,e_3,e_4\rangle = \R^4.$ Let $\ell_i$ be the $2$-subspace of $V_3$ that is orthogonal to $e_i,$ note that each $\ell_i$ is also a $2$-subspace of $V$ and so a triple $p=(\ell_1,\ell_2,\ell_3)$ is a point in $\mathcal H,$ we define $g:\Flag(\R^4)\ra \mathcal H$ to be the map that takes a flag $(0=V_0<V_1<V_2<V_3<V=\R^4)$ to this point $p$.
	
	Note that both $f$ and $g$ are well-defined. Now $f\circ g:\Flag(\R^4)\ra \mathcal H\ra \Flag(\R^4)$ is the identity map on $\Flag(\R^4)$ and the composition $g\circ f: \mathcal H\ra \Flag(\R^4)\ra \mathcal H$ first maps a triple $(\ell_1,\ell_2,\ell_3)$ with respective orthogonal vectors $(e_1,e_2,e_3)$ in the $3$-subspace containing all $\ell_i$ to the corresponding complete flag $(0=V_0<V_1=\langle e_1\rangle < V_2=\langle e_1,e_2\rangle < V_3=\langle e_1,e_2,e_3\rangle < V = \R^4)$ which is then mapped to $(\ell_1'=\ell_1,\ell_2',\ell_3')$ where $\ell_i'$ is orthogonal to $e_i'$ in $V_3,$ $e_2'$ is $e_2$ rotated along along a vector in $V_3$ orthogonal to $V_2,$ and similarly $e_3'$ can be obtained from $e_3$ via a rotation in $V_3,$ so the composition $g\circ f$ is homotopic to the identity on $\mathcal H$. Hence, $\mathcal H$ and $\Flag(\R^4)$ are homotopy equivalent.
\end{proof}

Note that in the complex case there is a fibration $$\Flag(\C^{n-1})\ra \Flag(\C^n)\xra{p_{n-1}} \Gr(n-1,\C^n)\cong \bb{CP}^{n-1}$$ where $p_{n-1}$ is a projection map, namely $p_k:\Flag(\C^{n+k})\ra \Gr(k,\C^{n+k})$ is defined by mapping a point $(0<V_1<\cdots < V_{n+k})$ in $\Flag(\C^{n+k})$ to $V_k,$ and the map $p$ induces an injection in cohomology $p^*:H^*(\Gr(k,\C^{n+k}))\ra H^*(\Flag(\C^{n+k}))$ \cite{5}, and by a theorem of Borel we can write an isomorphism of the cohomology of $\Flag(\C^n)$ with the coinvariant algebra. 

In \cite{6}, Kocherlakota gives an algorithm for the calculation of integral homology groups of real flag manifolds, but in the spirit of Theorem \ref{thm1} and Theorem \ref{thm2}, let us explicitly calculate the homology groups of $\mathcal H\simeq \Flag(\R^4)$.

\begin{thm}\label{thm4}
	The integral homology groups of $\mathcal H\cong \Flag(\R^4),$ are as follows. $$H_0(\mathcal H)\cong H_6(\mathcal H)\cong\Z, H_1(\mathcal H)\cong H_4(\mathcal H)\cong (\Z/2\Z)^3,$$ $$H_2(\mathcal H)\cong (\Z/2\Z)^2, H_3(\mathcal H)\cong \Z^2\times (\Z/2\Z)^2,$$ and for any other subscript homology is trivial.
\end{thm}
\begin{proof}
	By the transitive action of the general linear group on the space of complete flags in $\R^4$, we see that $\Flag(\R^4)$ is $\GL(4,\R)/U$ where $U$ is the maximal closed torus consisting of upper triangular matrices. This space is homotopy equivalent to $\Ort(4,\R)/(T\cap \Ort(4,\R))$ where $T\cap \Ort(4,\R)$ consists of matrices with $\pm 1$ on the main diagonal. 
	
	The Lie group $\Ort(4,\R)$ has two connected components and there is an element of $T\cap \Ort(4,\R)$ with negative determinant that can take an element of the connected component consisting of orthogonal matrices with determinant $-1$ to the other connected component $\SO(4,\R)$, hence, we have that $\Ort(4,\R)/(T\cap \Ort(4,\R))$ is diffeomorphic to $\SO(4,\R)/(T\cap \SO(4,\R))$ where $T\cap \SO(4,\R)$ is isomorphic to the elementary abelian group of order $8,$ $E_8,$ also isomorphic to $(\Z/2\Z)^3.$ 
	
	The universal covering group of $\SO(4,\R)$ is $\Spin(4)\cong \SU(2)\times \SU(2)\cong S^3\times S^3$ with the kernel of the universal cover map $\pi:\Spin(4)\ra \SO(4)$ being $Q:=\ker \pi \cong \Z/2\Z.$ Let $K^*$ be the lifting of $K:=T\cap \SO(4,\R)$ to $\Spin(4),$ it is isomorphic to $D_{8}\times \Z/2\Z$. We want to find the homology of $\Spin(4)/ K^*\simeq \SO(4,\R)/K.$ Since the defining equivalence relation acts smoothly and freely on $\Spin(4),$ the dimension of the quotient is $\dim \SO(4)-\dim K^* = 4\cdot 3 / 2 - 0 = 6.$ Since $\SO(4)$ is compact and connected, the quotient is compact and connected. The quotient is also orientable because $\SO(4)$ is closed, orientable and the action of the discrete subgroup $K^*$ is orientation preserving. Thus, $H_0(\mathcal H)\cong H_6(\mathcal H)\cong \Z.$ 
	
	Note that $\Spin(4)\simeq S^3\times S^3$ has integral homology $H_0(\Spin(4))\cong H_6(\Spin(4))\cong\Z$, $H_3(\Spin(4))\cong \Z^2$ and other homology groups are trivial. By Hurewicz theorem, $H_1(\mathcal H)\cong H_1(\Spin(4)/K^*)$ is isomorphic to the abelianization of $\pi_1(\Spin(4)/K^*)\cong K^*\cong D_8\times \Z/2\Z$ which is isomorphic to $(\Z/2\Z)^3$. By Poincaré duality $H_5(\mathcal H)$ is isomorphic to $H^1(\mathcal H),$ and by the universal coefficients theorem for cohomology we have an exact sequence $$0\ra \Ext^1(H_0(\mathcal H,\Z),\Z)\ra H^1(\mathcal H,\Z)\ra \Hom(H_1(\mathcal H,\Z),\Z)\ra 0$$ where $\Ext^1(H_0(\mathcal H,\Z),\Z)$ is trivial and $\Hom(H_1(\mathcal H,\Z),\Z)$ is also trivial, so $H_5(\mathcal H,\Z)\cong H^1(\mathcal H,\Z)$ is trivial. Again by universal coefficients theorem for cohomology we have $\Ext^1(H_4(\mathcal H,\Z),\Z)\cong (\Z/2\Z)^3$ and since $H_4(\mathcal H)$ is finitely generated and since it has no free part, we must have $H_4(\mathcal H)\cong (\Z/2\Z)^3.$ 

	Let $\pi:\mathcal H\ra \bb{RP}^3$ be defined as follows. Consider $\bb{RP}^3$ as the space of projective $2$-planes. A triple $p=(\ell_1,\ell_2,\ell_3)\in \mathcal H$ maps to the unique plane defined by $p,$ $\pi(p)\cong \bb{RP}^2,$ the space $\mathcal H$ is a fiber bundle over $\bb{RP}^3.$ A fiber of this bundle over a projective $2$-plane, a point $A\cong \bb{RP}^2$ in $\bb{RP}^3$ is the space of all triples of lines in $A$ such that any line in a triple intersects the two others at distinct points, hence a fiber is precisely homeomorphic to $\mathcal H'$ whose homology we already calculated in Theorem \ref{thm2}. 
	
	Since the homology of the base space and the fiber are known, and since the base space $\bb {RP}^3$ is path-connected, admits the structure of a CW-complex, and $\pi_1(\bb{RP}^3)$ acts trivially on $H^*(\mathcal H',\Z),$ there is a Serre spectral sequence of our fiber bundle $\pi:\mathcal H\ra \bb{RP}^3$. The cohomology of $\mathcal H'$ is $H^0(\mathcal H',\Z)=H^3(\mathcal H',\Z)=\Z$, $H^2(\mathcal H',\Z)=V_4$ and any other cohomology group is trivial. The cohomology of $\mathbb{RP}^3$ is $H^0(\bb{RP}^3,\Z)=H^3(\bb{RP}^3,\Z)=\Z,$ $H^2(\bb{RP}^3,\Z)=\Z/2\Z$ and any other cohomology group is trivial. 
	
	Hence, by $E_2^{p,q}=H^p(\bb{RP}^3,H^q(\mathcal H',\mathbb Z)),$ the second page $E_2$ of the Serre spectral sequence corresponding to our fiber bundle is as below. $$\begin{matrix} \mathbb Z & 0 & \mathbb Z/2\mathbb Z & \mathbb Z \\ V_4 & V_4 & V_4 & V_4 \\ 0 & 0 & 0 & 0 \\ \mathbb Z & 0 & \mathbb Z/2\mathbb Z & \mathbb Z\end{matrix}$$ The only non-trivial differential is $d_2^{0,3}:E^{0,3}_2\rightarrow E_2^{2,2}$ and since it is a homomorphism from $\Z$ to $V_4$ it cannot be surjective, hence the index $(2,2)$ survives in the $E_3$ page. By degree reasons $E_\infty$ degenerates to $E_3$ where every term in $E_3$ is same as that of $E_2$ because by the universal coefficients theorem for cohomology and the Poincaré duality, $\Ext^1(H_3(\mathcal H,\Z),\Z)\cong H_2(\mathcal H,\Z),$ so we must have $H_3(\mathcal H,\Z)\cong H^3(\mathcal H,\Z)\cong \Z^2\times V_4 \cong \Z^2\times (\Z/2\Z)^2$ and $H_2(\mathcal H,\Z)\cong H^4(\mathcal H,\Z)\cong (\Z/2\Z)^2$.
\end{proof}

\section{The space of triples of projective lines of distinct intersections in $\bb{RP}^n$}\label{sec4}
The proof of Theorem \ref{thm3} can be generalized to the space of all triples of projective lines of distinct intersections in $\bb{RP}^n.$ 

\begin{thm}\label{thm5}
	The space of all triples of projective lines in $\bb{RP}^{n}$ such that any line in a triple intersects the two others at distinct points is homotopy equivalent to the real partial flag variety $\Flag(1,2,3,\R^{n+1})$ consisting of partial flags of signature $(0,1,2,3,n+1)$.
\end{thm}
\begin{proof}
	The difference with the proof of Theorem \ref{thm3} is that we cannot complete a basis of a three-dimensional subspace of $V=\R^n$ uniquely to get a complete flag in $\Flag(\R^{n+1}),$ but instead given a triple of projective lines, we can construct a unique partial flag of signature $(0,1,2,3,n+1).$
	
	Let $\mathcal M$ be the space of all triples of projective lines of distinct intersections. Let $p=(\ell_1,\ell_2,\ell_3)$ be a point in $\mathcal M.$ Each $\ell_i$ corresponds to a $2$-subspace of $V=\R^{n+1},$ and $\ell_1\cup \ell_2\cup \ell_3$ is contained in some $3$-subspace $V'$ of $V.$ Let $e_i$ be a vector orthogonal to $\ell_i$ in $V',$ and let $V_1=\langle e_1 \rangle,$ $V_2=\langle e_1,e_2 \rangle,$ $V_3=\langle e_1,e_2,e_3\rangle,$ and so $(0=V_0<V_1<V_2<V_3<V)$ is a partial flag in $\R^{n+1}$ and we define $f:\mathcal M \ra \Flag(1,2,3,\R^{n+1})$ to be the map that takes $p$ to this partial flag. 
	
	Given a partial flag $(0=V_0<V_1<V_2<V_3<V=\R^{n+1})$ with signature $(0,1,2,3,n+1),$ we can construct an orthonormal basis $(e_1,e_2,e_3)$ unique up to multiplying each basis element by a unit such that $0=V_0<V_1=\langle e_1\rangle < V_2=\langle e_1,e_2\rangle < V_3=\langle e_1,e_2,e_3\rangle < V.$ Let $\ell_i$ be the $2$-subspace of $V_3$ which is orthogonal to $e_i.$ Note that each $\ell_i$ is also a $2$-subspace of $V$ and so a triple $p=(\ell_1,\ell_2,\ell_3)$ is a point in $\mathcal M.$ We define $g:\Flag(1,2,3,\R^4)\ra \mathcal M$ to be the map that takes a flag $(0=V_0<V_1<V_2<V_3<V=\R^{n+1})$ to this point $p$. Both $f$ and $g$ are well-defined. 
	
	Now $f\circ g:\Flag(1,2,3,\R^4)\ra \mathcal M\ra \Flag(1,2,3,\R^4)$ is the identity map on $\Flag(1,2,3,\R^4)$ and the composition $g\circ f: \mathcal M\ra \Flag(1,2,3,\R^4)\ra \mathcal M$ first maps a triple $(\ell_1,\ell_2,\ell_3)$ with respective orthogonal vectors $(e_1,e_2,e_3)$ in the $3$-subspace containing all $\ell_i$ to the corresponding complete flag $(0=V_0<V_1=\langle e_1\rangle < V_2=\langle e_1,e_2\rangle < V_3=\langle e_1,e_2,e_3\rangle < V)$ which is then mapped to $(\ell_1'=\ell_1,\ell_2',\ell_3')$ where $\ell_i'$ is orthogonal to $e_i'$ in $V_3,$ and the orthonormal triple $(e_1',e_2',e_3')$ is obtained from $(e_1,e_2,e_3)$ via some rotation, and so the composition $g\circ f$ is homotopic to the identity on $\mathcal M$. Hence, $\mathcal M$ and $\Flag(1,2,3,\R^4)$ are homotopy equivalent.
\end{proof}

Since there is an algorithm by Kocherlakota \cite{6} to compute the integral homology of real flag manifolds, this particular projective configuration space is well-understood in terms of (co)homology.
		
\section{Acknowledgement}
The author is grateful to his advisor at HSE University, Viktor Vassiliev, for recommending the problem and for helpful discussions.

 \Address
\end{document}